\documentclass[12pt]{amsart}

\usepackage{amssymb,amscd,amsthm,verbatim,amsmath,color,fancyhdr,mathrsfs, bm}
\usepackage{graphicx}
\usepackage{turnstile}

\numberwithin{equation}{section}

\usepackage[letterpaper, left=2.5cm, right=2.5cm, top=2.5cm,bottom=2.5cm,dvips]{geometry}

\usepackage{style} 

\newtheorem{theorem}{Theorem}[section]
\newtheorem{proposition}[theorem]{Proposition}
\newtheorem{lemma}[theorem]{Lemma}
\newtheorem{remark}[theorem]{Remark}
\newtheorem{corollary}[theorem]{Corollary}

\newtheorem{definition}[theorem]{Definition}

\DeclarePairedDelimiterX{\bracket}[3]{#1}{#2}{#3}

\providecommand{\newoperator}[3]{\newcommand*{#1}{\mathop{#2}#3}}
\providecommand{\renewoperator}[3]{\renewcommand*{#1}{\mathop{#2}#3}}

\renewoperator{\Re}{\mathrm{Re}}{\nolimits}
\renewoperator{\Im}{\mathrm{Im}}{\nolimits}

\DeclarePairedDelimiterXPP{\nrm}[2]{}{\lVert}{\rVert}{\ensuremath{_{#1}}}{\ifblank{#2}{\:\cdot\:}{#2}}
\newcommand{\norm}[2]{\nrm*{#1}{#2}}

\newcommand{\abs}[1]{\bracket*{\lvert}{\rvert}{#1}}

\newcommand{\inner}[1]{\bracket*{\langle}{\rangle}{#1}}

\DeclarePairedDelimiterXPP\prob[1]{\mathbb{P}}{\lbrace}{\rbrace}{}{#1} 

\DeclarePairedDelimiterXPP\probability[2]{\mathbb{P}_{#1}}{\lbrace}{\rbrace}{}{#2} 

\DeclarePairedDelimiterXPP\expectation[1]{\mathbb{E}}{\lbrack}{\rbrack}{}{#1} 

\DeclarePairedDelimiterXPP\expectationdist[2]{\mathbb{E}_{#1}}{\lbrack}{\rbrack}{}{#2} 

\DeclarePairedDelimiterXPP\variance[1]{\mathrm{Var}}{\lbrack}{\rbrack}{}{#1} 

\DeclarePairedDelimiterXPP\variancedist[2]{\mathrm{Var}_{#1}}{\lbrack}{\rbrack}{}{#2} 

\DeclarePairedDelimiterXPP\covariance[2]{\mathrm{Cov}}{(}{)}{}{#1,\mathopen{}#2} 

\newoperator{\supp}{\mathrm{supp}}{\nolimits}

\makeatletter
\providecommand*{\diff}
{\@ifnextchar^{\DIfF}{\DIfF^{}}}
\def\DIfF^#1{
	\mathop{\mathrm{\mathstrut d}}
	\nolimits^{#1}\gobblespace}
\def\gobblespace{
	\futurelet\diffarg\opspace}
\def\opspace{
	\let\DiffSpace\!
	\ifx\diffarg(
	\let\DiffSpace\relax
	\else
	\ifx\diffarg[
	\let\DiffSpace\relax
	\else
	\ifx\diffarg\{
	\let\DiffSpace\relax
	\fi\fi\fi\DiffSpace}

\providecommand*{\pdiff}
{\@ifnextchar^{\pDIfF}{\pDIfF^{}}}
\def\pDIfF^#1{
	\mathop{\mathrm{\mathstrut \partial}}
	\nolimits^{#1}\gobblespace}
\def\gobblespace{
	\futurelet\diffarg\opspace}
\def\opspace{
	\let\DiffSpace\!
	\ifx\diffarg(
	\let\DiffSpace\relax
	\else
	\ifx\diffarg[
	\let\DiffSpace\relax
	\else
	\ifx\diffarg\{
	\let\DiffSpace\relax
	\fi\fi\fi\DiffSpace}

\makeatother

\DeclarePairedDelimiterX\Set[1]\{\}{
	
	#1
}

\newcommand{\Complex}{\mathbb{C}}

\newcommand{\T}{\mathbb{T}}

\newcommand{\Bcal}{\mathcal{B}}

\newcommand{\Fcal}{\mathcal{F}}

\newcommand{\Hcal}{\mathcal{H}}

\newcommand{\Mcal}{\mathcal{M}}

\newcommand{\Ucal}{\mathcal{U}}

\newcommand{\eq}{\begin{equation}}
\newcommand{\en}{\end{equation}}

\usepackage{amsrefs} 

\title[Douglas-Rudin Approximation]{Douglas-Rudin Approximation theorem for operator-valued functions on the unit ball of $\mathbb{C}^d$ }

\author[Kumar]{Poornendu Kumar}
\address{Kumar \\ Department of Mathematics \\ University of Manitoba\\ Winnipeg, R3T 2N2, Canada\\ {Email: poornendu.kumar@umanitoba.ca }}

\author[Rastogi]{Shubham Rastogi}
\address{Rastogi\\ Department of Mathematics \\
Indian Institute of Technology Bombay\\ Powai, Mumbai, 400076, India\\ {Email: shubhamr@math.iitb.ac.in }}

\author[Tripathi]{Raghavendra Tripathi}
\address{Tripathi\\ Department of Mathematics \\ University of Washington\\ Seattle WA 98195, USA\\ {Email: raghavt@uw.edu}}

\keywords{Unimodular functions, Inner functions, Approximation, Operator-valued functions}
	
\subjclass[2010]{ (Primary) 46E40, (Secondary) 32A99, 30J05}

\thanks{The first author is partially supported by a PIMS postdoctoral fellowship.}

\date{\today}

\begin{document}

\begin{abstract}

Douglas and Rudin proved that any unimodular function on the unit circle $\T$ can be uniformly approximated by quotients of inner functions. We extend this result to the operator-valued unimodular functions defined on the boundary of the open unit ball of $\mathbb{C}^d$. Our proof technique combines the spectral theorem for unitary operators with the Douglas-Rudin theorem in the scalar case to bootstrap the result to the operator-valued case. This yields a new proof and a significant generalization of Barclay's result [Proc. Lond. Math. Soc. 2009] on the approximation of matrix-valued unimodular functions on $\T$. 
\end{abstract}

\maketitle
\section{Introduction}\label{sec:Intro}

Inner functions are one of the most ubiquitous objects in the study of function theory of domains. Inner functions were introduced by R. Nevanlinna. Following the groundbreaking contributions of the Riesz brothers, Frostmann, and Beurling, they have evolved into a fundamental concept in analysis (see for examples~\cites{Gar, Hari}) and continue to inspire research even today. It is a well-known result, due to Fatou, that any bounded holomorphic function $f$ on the open unit disc $\mathbb D 
=\{z\in\mathbb{C}: |z|<1\}$, admits radial limits $f^{*}(e^{i\theta})\coloneqq \lim_{r \to 1} f(re^{i\theta})$ almost everywhere on the unit circle $\mathbb{T}$ with respect to the Lebesgue measure on $\mathbb{T}$. If the boundary function $f^{*}$ has modulus one almost everywhere on the circle with respect to the Lebesgue measure on $\mathbb{T}$, then $f$ is called an inner function. Whenever there is no scope for confusion, we identify a bounded analytic function $f$ defined on the disc with the boundary function $f^{*}$ defined on the circle via the radial limit and we use the same symbol, namely $f$, to denote both. The notion of inner functions can be naturally generalized to matrix-valued settings and operator-valued settings as well-- where the modulus one condition is replaced by taking values in the space of unitary matrices or isometries. 

In this paper, our focus lies in exploring the approximation properties of inner functions. Carathéodory~\cite{Car}, in his study of holomorphic functions on the open unit disc, demonstrated that any holomorphic self-map on $\mathbb{D}$ can be uniformly approximated by rational inner functions over compact subsets of $\mathbb{D}$. Since then this landmark result has sparked numerous generalizations across various domains, as evidenced by works~\cites{ABJK2023, Rudin, Ale, BBK2024}. On the other hand, Fisher~\cite{Fisher} showed that any function defined on the disc-- continuous up to its boundary--can be uniformly approximated through convex combinations of rational inner functions. For an insightful and comprehensive account of approximation by inner functions, we recommend the survey by Mashreghi and Ransford~\cite{mashreghi2018approximation}. Our focus is on the unimodular functions and their approximation by inner functions. In this vein, of particular interest to us is the Douglas-Rudin Theorem~\cite{RudinDouglas}, which we state below.


\begin{theorem}[Douglas-Rudin]\label{thm:DR}
    Let $f:\T\to \Complex$ be a unimodular function, that is, $f$ is measurable and $|f(z)|=1$ for a.e. $z\in \T$. Then, for every $\epsilon>0$ there exist inner functions (even Blaschke products) $\varphi, \psi$ on $\mathbb{D}$ such that 
    \begin{align*}
         \operatorname{ess\hspace{1mm}sup}_{z\in \T}|f(z)-\varphi(z)\psi(z)^{-1}|<\epsilon.
    \end{align*}
   
\end{theorem}
\noindent Moreover, Helson and Sarason \cite{HS} showed that if the function $f$ is continuous, then one can choose $\varphi$ and $\psi$ to be rational inner functions. This theorem has several important consequences, many of which were proved in~\cite{RudinDouglas} itself. As a far-reaching consequence of this result, Axler~\cite{Axler} proved that any arbitrary bounded measurable function on the circle can be elegantly expressed by the quotient of two nice functions: one sourced from $H^{\infty}+ C(\mathbb{T})$, and the other characterized as a Blaschke product on the unit disc.

Moving to the higher variable setting, there are two domains in $\mathbb{C}^d$ that can be considered analogous to the disc in $\mathbb{C}$, namely the open unit polydisc $\mathbb{D}^d$ and the complex Euclidean open unit ball 
$$\mathbb{B}_d=\{z\in\mathbb{C}^d: \|z\|<1\}.$$
The notion of inner functions and unimodular functions naturally extends to these domains. More precisely, let $\Omega = \mathbb{D}^d$ or $\mathbb{B}_d$ and $b\Omega$ be the distinguished boundary of $\Omega$. Then we say that a holomorphic map $f$ on $\Omega$ is inner if its radial limits are of modulus one on $b\Omega$ almost everywhere with respect to the Lebesgue measure on $b\Omega$. The inner functions in these domains are considerably more complicated~\cites{Rudin, RudinBall}. While the structure of rational inner functions and the Carath\'eodory approximation result are understood in the polydisc case, the existence of non-constant inner functions on $\mathbb{B}_{d}$ for $d\geq 2$ remained a long-standing problem until resolved by Aleksandrov~\cite{Ale}. Furthermore, he showed that the set of inner functions is dense in the uniform compact-open topology on the open unit ball, i.e., a version of Carath\'eodory's theorem holds.

Using the cohomology group theory and the structure of rational inner functions, McDonald~\cite{Mc} extended the Douglas-Rudin theorem to the polydisc. However, McDonald's result~\cite{Mc} only allows the uniform approximation of \emph{continuous} unimodular functions -- albeit by rational inner functions. The problem of approximating measurable unimodular functions on polydisc still seems to be an open problem. The case of unimodular functions on the spheres 
$$\mathbb{S}_d=\{z\in\mathbb{C}^d: \|z\|=1\}$$ 
is more delicate. Rudin~\cite{RudinPaper} proved that an unimodular function on $\mathbb{S}_d$ (for $d\geq 2$) cannot be approximated by quotients of inner functions in the uniform norm. However, he proved that any unimodular function on $\mathbb{S}_d$ can indeed be approximated by quotients of inner functions in almost everywhere sense.

Given the success of the previous results, the natural progression is to explore their generalization in the matrix-valued or operator-valued setting. For matrix-valued unimodular functions on $\mathbb{D}$, such result is known due to~\cite{BS}*{Section 3}. The proof technique in~\cite{BS}, however, has limited applicability due to its reliance on estimating certain inner functions and the finite-dimensional nature of the problem, along with the consideration of single variables. To prove such results for $\mathbb{B}_d$ or extending~\cite{BS} to the operator-valued setting, an alternative approach to their proof is needed.
In this paper, we address precisely this need by extending the Douglas-Rudin result for the operator-valued functions defined on $\mathbb{B}_d$.

Our proof involves reducing the problem to the unimodular functions which assume only two distinct values-- one of which being the identity operator. We combine the Spectral theorem and use the approximation result for scalar-valued functions to obtain our result. As the Douglas-Rudin theorem for scalar-valued functions is known for the open unit ball, our approach readily extends to this case. 

The statement of the generalized Douglas-Rudin theorem requires some definitions and notations. In the following section, we provide the necessary definitions and notations before stating our main results.

\section{Main results}
Throughout this paper, $\mathcal{H}$ denotes a separable Hilbert space over $\mathbb{C}$. We denote by $\mathcal{B}(\mathcal{H})$ the set of all bounded linear operators on $\mathcal{H}$, and $\Ucal(\Hcal)$ denotes the set of all unitary operators on $\Hcal$. A map $f:\Omega \rightarrow \mathcal{B}(\mathcal{H})$ is said to be holomorphic if its Fréchet derivative exists for all $z\in \Omega$. This definition is equivalent to the following characterization: for any pair of vectors $h_1, h_2 \in \mathcal{H}$, the scalar-valued function
\[
\Omega \ni z \mapsto \langle f(z)h_1, h_2\rangle
\]
is holomorphic. For further elucidation, refer to \cite{Game} and a comprehensive survey for operator-valued holomorphic functions \cite{ES}. The notion of unimodular functions in the operator-valued case can be generalized in the following way:

\begin{definition}
A measurable function $f:b\Omega \to \Bcal(\Hcal)$ is said to be \emph{unimodular} if $f(z)$ is unitary for almost every $z\in b\Omega$.
\end{definition}
Such functions have been greatly studied by several authors and found a host of applications in analysis. To mention a few, these functions play an important role in the study of Toeplitz operators, convolution integral operators, and singular integral operators, in particular, in the study of the Fredholm properties of such operators. We refer the reader to~\cites{CG,FB,FK} for more details.

We denote by $H^\infty(\Omega, \Bcal(\mathcal{H}))$ the set of holomorphic maps $f$ on $\Omega$ for which $$\sup_{z\in \Omega}\|f(z)\|_{\Bcal(\mathcal{H})} < \infty$$ and $L^\infty(b \Omega, \mathcal{B}(\mathcal{H}))$ denotes the set of essentially bounded measurable functions on the distinguished boundary of $\Omega$.
See~\cites{NF, Hoff} for more details on operator-valued holomorphic functions.
\begin{definition}
 A function $\varphi\in H^\infty(\Omega, \Bcal(\mathcal{H}))$ is said to be \emph{two-sided inner} if
$$\varphi(z)^*\varphi(z)=I_{\Hcal}= \varphi(z)\varphi(z)^* \quad \text{ for almost every z on } b\Omega.$$
\end{definition}
If the first equality holds, then $\varphi$ is said to be \emph{inner}. Note that if $\mathcal{H}$ is finite-dimensional, then the notion of inner functions and two-sided inner functions coincide. In this case, the structure of {\em rational} inner functions is well known. In a seminal work by Potapov (\cite{Potapov}), he proved that any $N \times N$ matrix-valued rational inner function $\Phi$ can be represented as:

\begin{align}\label{Matrix-Inner}
  \Phi(z) = U \left( \prod_{m=1}^{M} \left( b_{\alpha_m}(z)P_m + (I_{\mathbb{C}^N} - P_m) \right) \right)  
\end{align}
where $z \in \mathbb{D}$, $m$ is a natural number, $U$ is an $N \times N$ unitary matrix, $P_m$ are projections onto certain subspaces of $\mathbb{C}^N$, $\alpha_m$ are points in the open unit disc, and $$b_{\alpha}(z) = \frac{z - \alpha}{1 - \bar{\alpha}z}$$ for $\alpha \in \mathbb{D}$ represents a Blaschke factor. These functions are commonly referred to as Blaschke-Potapov products. Each Blaschke-Potapov factor can also be written as 
$$U_1\begin{bmatrix}
   b_{\alpha_m}(z)I_{r} & 0 \\
   0 & I_{N-r}
\end{bmatrix} U_2$$
for some unitaries $U_1$ and $U_2$. Recently, Curto, Hwang, and Lee~\cites{curto2022operator, Curto2} have extensively studied two-sided inner functions along with establishing some relations with Hankel operators. They generalized the Potapov result for operator-valued rational two-sided inner functions; see~\cite{curto2022operator} for more details on this.

Recall that a sequence of operators $T_n$ on $\mathcal{H}$ converges to $T$ on $\mathcal{H}$ in the \emph{weak-sense} if for any $h, k \in \mathcal{H}$, 
\[
\langle (T_n - T)h, k\rangle_{\mathcal{H}} \rightarrow 0\;,
\]
as $n\to \infty$. Let $f_n:b\Omega\to \mathcal{B}(\mathcal{H})$ be a sequence of bounded measurable functions. We say that $f_n$ converges, \emph{uniformly on $b\Omega$ in the $weak$ topology}, to $f:b\Omega\to \mathcal{B}(\mathcal{H})$ if for every $h, k\in \Hcal,$ we have
\[
\operatorname{ess\,sup}_{z\in b\Omega}\abs{\langle (f_n(z) - f(z))h, k\rangle} \rightarrow 0.
\]

We now state our first main result that generalizes Douglas-Rudin Theorem~\ref{thm:DR} to the operator-valued functions (in the single-variable case).

\begingroup
\setcounter{theorem}{0}
\renewcommand\thetheorem{\Alph{theorem}}
\begin{theorem}\label{Thm:Main_Theorem}
Any unimodular function $f:\T\to \Bcal(\Hcal)$ can be approximated uniformly on $\T$, in weak topology, by the quotients of two-sided inner functions. Furthermore, if $f$ is continuous (in norm topology), the uniform approximation holds in the operator norm.
\end{theorem}
\endgroup
Note that when $\mathcal{H}$ is finite-dimensional, the weak topology coincides with the operator norm topology. Thus, the result~\ref{Thm:Main_Theorem} can be strengthened. That is, we achieve uniform approximation in the norm topology--without requiring any continuity assumption. Thus, we recover the following result due to Barclay~\cite{BS}*{Section 3} that we record below as a corollary. 
\begin{corollary}
\label{Cor:Bar}
    Let $f:\T\to\Mcal_n(\mathbb{C})$ be a unimodular function. Then, for every $\epsilon>0$, there exists matrix-valued inner functions $\Phi$ and $\Psi$ such that 
    \[\operatorname{ess\hspace{1mm}sup}_{z\in \T}\|f(z) - \Phi(z)\Psi(z)^{*}\|_{op}< \epsilon.\]
\end{corollary}

While the proof in~\cite{BS} relies on estimating specific inner functions and leveraging the structure of certain functions, our proof differs significantly. Our proof technique relies on the spectral theorem for unitary operators along with the result in the scalar-valued case. This makes our proof technique very versatile. In particular, as a version of the Douglas-Rudin theorem is known for the unimodular functions $f: \mathbb{S}_{d} \rightarrow \mathbb{C}$ (see ~\cite{RudinBall}), our proof technique enables us to extend this result to the operator-valued unimodular functions on \( \mathbb{S}_d \) for all $d\geq 1$. We must emphasize that since the approximation of scalar-valued unimodular functions on \( \mathbb{B}_d \), where \( d \geq 2 \), holds only in an almost everywhere (a.e.) sense (see~\cite{RudinBall}*{Section 6}), the Douglas-Rudin type approximation result for operator-valued unimodular functions also holds only in the almost everywhere (a.e.) sense for \( d \geq 2 \).

Before we state our result in this setting, we need another definition. We say that $f_n:b\Omega\to \mathcal{B}(\Hcal)$ converges, \emph{in the $weak$ topology}, to $f:b\Omega\to \mathcal{B}(\mathcal{H})$ a.e., if for every $h, k\in \Hcal,$ we have
\[
\langle (f_n(z) - f(z))h, k\rangle \rightarrow 0, \quad \text{for a.e. } \quad z\in b\Omega\;.
\]

\begingroup
\setcounter{theorem}{1}
\renewcommand\thetheorem{\Alph{theorem}}
\begin{theorem}\label{Thm:MainTheorem_Ball2}
Let $f:\mathbb{S}_d\to \Bcal(\Hcal)$ be a unimodular function. Then, $f$ can be approximated in the weak topology by quotients of two-sided inner functions a.e. Additionally, if $f$ is continuous, the above approximation holds in operator norm a.e.
\end{theorem}
\endgroup

Following~\cite{RudinBall}, it is clear that our result is optimal in this case. That is, without additional assumptions on $f$, the above result can not be strengthened to uniform convergence.

\subsection{Outline of the paper}
In Section~\ref{sec:proofs}, we first collect some preliminary lemmas that allow us to reduce the proof of Theorem~\ref{Thm:Main_Theorem} and Theorem~\ref{Thm:MainTheorem_Ball2} to approximating unimodular functions taking only two values. We give the proof of our main results in Section~\ref{subsec:Mainproof} and Section~\ref{sec:Thm2}. We end with some remarks and applications of our results in Section~\ref{sec:conclusion}.

\section{Preliminaries}\label{sec:proofs}
The proof of Theorem~\ref{Thm:Main_Theorem} and Theorem~\ref{Thm:MainTheorem_Ball2} require some common reductions. In this section, we state and prove some preliminary lemmas for this task. Lemma~\ref{lem:ApproximationBySimpleUnimod} is essentially a topological result. This reduces our problem to approximating an unimodular function taking only finitely many values. In Lemma~\ref{lem:BinaryApprox}, we write an unimodular function taking only finitely many values as a finite product of unimodular functions taking only two values. This reduces our problem to approximating unimodular functions taking only two values. Lemma~\ref{lem:ApproximationBySimpleUnimod} requires the range of our unimodular function to be precompact. Therefore, for an arbitrary unimodular function, we get the approximation in the weak topology as $\Ucal(\Hcal)$ is precompact in the weak topology. However, if the range of the function is compact in a stronger topology--for instance if $f$ is continuous with respect to operator norm-- then Lemma~\ref{lem:ApproximationBySimpleUnimod} gives approximation in the stronger topology. 
\begin{lemma}
\label{lem:ApproximationBySimpleUnimod}
 Let $X, Y$ be two metric spaces. Let $f: X \to Y$ be a Borel measurable map such that $\overline{f(X)}$ is compact. Then, for every $\epsilon>0$, there exists a function $g:X\to f(X)$ such that $g$ takes only finitely many values and $\sup_{x\in X}d_{Y}(f(x), g(x))<\epsilon$.
\end{lemma}
\begin{proof}
Let $B(y, r)$ denote the open ball of radius $r$ centered at $y\in Y$. Note that $\overline{f(X)}\subseteq \cup_{x\in X}B(f(x), \epsilon)$, that is, $\{B(f(x), \epsilon): x\in X\}$ is an open cover of $\overline{f(X)}$. Since $\overline{f(X)}$ is compact, we obtain a finite refinement, that is, there exists a finite collection of points $\{x_1, \ldots, x_m\}\subseteq X$ such that $$f(X)\subseteq \bigcup_{i=1}^{m}B(f(x_i), \epsilon).$$

Define $E_1:=B(f(x_1), \epsilon)$ and for $2\leq j\leq m$ iteratively define $E_j:=B(f(x_j), \epsilon)\setminus \cup_{i<j} E_i$. Again define $g:X\to Y$ by setting $$g\vert_{f^{-1}(E_i)}=f(x_i).$$ Note that $$\sup_{x\in X}d_Y(f(x), g(x))\leq \max_{j}diam(E_j).$$ The proof is complete by observing that $\max_{j}diam(E_j)<2\epsilon$ by construction. 
\end{proof}

Following is an immediate corollary of Lemma~\ref{lem:ApproximationBySimpleUnimod} that we use later in the proof. Recall that $b\Omega$ denotes the distinguished boundary of $\Omega$. 
\begin{lemma}\label{lem:unitofinite}
Let $f:b\Omega\to \Bcal(\Hcal)$ be a unimodular function. Let $d$ be a metric on the closed unit ball of  $\Bcal(\Hcal)$ that metrizes the weak topology. Then, for every $\epsilon>0$, there exists a unimodular function $g:b\Omega\to \Bcal(\Hcal)$ taking only finitely many values such that 
    \begin{align}
    \label{eq:Approx}
        \sup_{z\in b\Omega} d(f(z), g(z))<\epsilon\;.
    \end{align}

    Furthermore, if $f:b\Omega\to \Bcal(\Hcal)$ is continuous in operator-norm topology then we can replace $d(f(z), g(z))$ in equation~\eqref{eq:Approx} by $\|f(z)-g(z)\|_{op}$. 
\end{lemma}
\begin{remark}
 Lemma~\ref{lem:ApproximationBySimpleUnimod} cannot hold in general when the range is non-compact (at least not without some continuity assumption on $f$). To see this, in operator valued case, let $\theta_0=0$ and $\theta_{n}=\sum_{i=1}^{n}2^{-i}$ and let $E_i\subseteq \T$ be such that $$E_i=\{e^{2i\pi\theta}\in \T: \theta_{n-1}\theta\in \theta_{n}\}.$$ Let $\Hcal$ be a separable infinite dimensional Hilbert space and fix a countable basis $\{e_{i}:i\in \mathbb{N}\}$ and define an operator $T_i$ as follows. For any $f\in \Hcal$, define $$T_if = f-2\inner{f, e_i}.$$ Note that $T_i$ is unitary operator and $\|T_i-T_j\|=1$ for $i\neq j$. Now define $f:\T\to \Ucal(\Hcal)$ as $f(z)=T_i$ for $z\in E_i$. Note that $f$ cannot be approximated by a simple function in operator norm (or even strong topology for that matter).    
\end{remark}

For our discussion in this paper, we will refer to the unimodular functions taking only finitely many values as \emph{simple unimodular functions}. Our next lemma shows that simple unimodular functions can be written as a product of unimodular functions taking only two values.
\begin{lemma}
\label{lem:BinaryApprox}
Let $f:b\Omega\to \Bcal(\Hcal)$ be a simple unimodular function. Then, there exists a finite collection of unimodular functions $\varphi_{i}:b\Omega\to \Ucal(\Hcal), i\in \Fcal$ with $|\Fcal|<\infty$ such that each $\varphi_i$ takes at most two values and
\[f(z)=\prod_{i\in \Fcal} \varphi_i(z),\]
for every $z\in b\Omega.$
\end{lemma}
\begin{proof}
    The proof is standard but we include it for completeness. Since $f$ takes only finitely many values, there exists a partition of $b\Omega$ by finitely many measurable subsets $E_1, \ldots, E_m$ such that $f\vert_{E_i}\equiv U_i\in \Ucal(\Hcal)$. For $i\in \{1, \ldots, m\}$, define unimodular functions $\varphi_i:b\Omega\to \Ucal(\Hcal)$ such that 
   \begin{equation*}
		\varphi_i=
		\begin{cases}
			U_i & \text{on }\   E_i; \\
			I & \text{ on }\   E_i^c
		\end{cases}
	\end{equation*}
 It is clear that $f(z)=\prod_{i=1}^{m}\varphi_i(z)$ for all $z\in b\Omega.$
\end{proof}
Given the above two lemmas, it suffices to prove our theorems for unimodular functions taking only two values. Without loss of generality, we can also assume that it takes two distinct values. Also note that it follows from the proof of Lemma~\ref{lem:BinaryApprox} that we can assume that the identity operator $I_{\Hcal}$ is in the range of $f$.  

To prove our results for such unimodular functions, we combine the spectral theorem for unitary operators with the approximation results for scalar-valued unimodular functions. As mentioned in the Introduction, these two cases need to be dealt with separately because the approximation of scalar-valued functions on $\mathbb{B}_d$ for $d\geq 2$ holds only in a.e. sense.

\section{Proof of Theorem \ref{Thm:Main_Theorem}} \label{subsec:Mainproof}
The next result is the most crucial step towards proving the main theorem.

\begin{proposition}
\label{prop:Key_Proposition2}
Let $E\subseteq \T$ be a measurable subset. Let $f:\T\to \Bcal(\Hcal)$ be a unimodular function such that $f\vert_{E}=I_{\Hcal}$ and $f\vert_{E^{c}}=T$ for some $T\in \Ucal(\Hcal)$. Then, for any $\epsilon>0$ there exist two-sided inner functions $\Phi, \Psi$ such that 
\[\operatorname{ess\hspace{1mm}sup}_{z\in \T}\|f(z)-\Phi(z)\Psi(z)^{-1}\|_{op}< \epsilon\;.\]
\end{proposition}
\begin{proof}

Let $f:\T\to \Bcal(\Hcal)$ be a unimodular function such that
\begin{equation*}
		f=
		\begin{cases}
			I_{\Hcal} & \text{on }\   E; \\
			T & \text{ on }\   E^c
		\end{cases}
	\end{equation*}
 for some measurable set $E\subseteq \T$ and a unitary operator $T$ on $\Hcal$. By spectral theorem, there exist a semifinite measure space $(\Sigma,\mu),$  a function $\eta\in L^{\infty}(\Sigma, \mu)$ with $$|\eta(x)|=1\quad \text{ a.e. }\quad  x\in \Sigma$$ and a unitary $\Gamma\colon \Hcal\to L^2(\Sigma, \mu)$ such that $T = \Gamma^{*}M_{\eta}\Gamma$ where $M_{\eta}$ denotes the multiplication by $\eta$, see \cite{Simon}*{Chapter 5}.

Let $\epsilon>0$ be fixed. Let $\widetilde{\eta}:\Sigma\to \Complex$ be a simple unimodular function such that $\|\widetilde{\eta}-\eta\|_{\infty}< \epsilon.$ Define $\widetilde{T}\coloneqq \Gamma^{*}M_{\widetilde{\eta}}\Gamma$ where $M_{\widetilde{\eta}}$ denote the multiplication by $\widetilde{\eta}$. Then, $\|T-\widetilde{T}\|_{op}<\epsilon$. In other words, the simple unimodular function $\widetilde{f}:\T\to \Ucal(\Hcal)$ defined as $$\widetilde{f}\vert_{E}=I_{\Hcal}\quad  \text{ and } \quad\widetilde{f}\vert_{E^c}=\widetilde{T}$$ satisfies $\operatorname{ess\hspace{1mm}sup}_{z\in \T}\|f(z)-\widetilde{f}(z)\|_{op}<\epsilon$.

We now approximate $\widetilde{f}$ by a quotient of two-sided inner functions. To this end, we define two two-sided inner functions $\Phi, \Psi:\mathbb{D}\to \Bcal(L^2(\Sigma, \mu))$ and show that $g:\T\to \Ucal(\Hcal)$ defined as $$g(z)=\Gamma^{*}\Phi(z) \Psi(z)^{*}\Gamma$$ satisfies $\operatorname{ess\hspace{1mm}sup}_{z\in \T}\|g(z)-\widetilde{f}(z)\|_{op}<\epsilon$ concluding our proof.

Let $\widetilde{\eta}=\sum_{j=1}^{m}\alpha_i\chi_{\Sigma_i}$ where $\Sigma_{i}$,  for $i=1, \ldots, m$, is a partition of $\Sigma$ into measurable subsets. For any $1\leq i\leq m,$ invoke Theorem \ref{thm:DR} to get scalar-valued inner functions $\varphi_i, \psi_i:\mathbb{D}\to \overline{\mathbb{D}}$  such that $$\theta_i\coloneqq \varphi_i/\psi_i$$ satisfies $|\theta_{i}\vert_{E}-1|<\epsilon$ and $|\theta_{i}\vert_{E^c}-\alpha_i|<\epsilon$ almost everywhere. Set 
\[\varphi \coloneqq \sum_{j=1}^{m}\varphi_i\chi_{\Sigma_i}:\mathbb{D}\to L^{\infty}(\Sigma, \mu)\quad  \text{ and }\quad  \psi \coloneqq \sum_{j=1}^{m}\psi_i\chi_{\Sigma_{i}}:\mathbb{D}\to L^{\infty}(\Sigma, \mu).\]

Define $\Phi, \Psi:\mathbb{D}\to \Bcal(L^2(\Sigma, \mu))$ as 
$$\Phi(z) \coloneqq M_{\varphi(z)}\quad  \text{ and }\quad  \Psi(z) \coloneqq M_{\psi(z)}.$$ Fix $f, g\in L^2(\Sigma, \mu)$ and consider the function $F:\mathbb{D}\to \mathbb{C}$ defined as 

$$F(z)\coloneqq \inner{f, \varphi(z) g}=\sum_{j=1}^{m}\inner{f, \varphi_j(z)\chi_{\Sigma_j}g}.$$

Note that $\abs{F(z)}\leq \norm{2}{f}\norm{2}{g}$. Therefore, $\sup_{z\in \mathbb{D}}\norm{op}{\Phi(z)}\leq 1$. Again set $$F_j(z)\coloneqq \inner{f, \varphi_j(z)\chi_{\Sigma_j}g}.$$Since $\varphi_{j}$ is holomorphic, it follows that $F_j$ is holomorphic for each $j=1,\ldots, m$. Thus, we conclude that $\Phi, \Psi\in H^{\infty}(\mathbb D, \Bcal(L^2(\Sigma, \mu)))$.
Since $\varphi_i$'s are inner, for a.e. $z\in \T,$ we have $$|\varphi(z)(x)|=1\quad  \text {for all } x\in \Sigma.$$Therefore $\Phi(z)\in \Ucal(\Hcal)$ for a.e. $z\in \T$. This completes the proof that $\Phi$ is a two-sided inner function. Similarly, we can show that  $\Psi$ is a two-sided inner function. It is now easily verified that $$\operatorname{ess\hspace{1mm}sup}_{z\in \T}\|g(z)-\widetilde{f}(z)\|<\epsilon$$ where $g(z)=\Gamma^{*}\Phi(z)\Psi(z)^{*}\Gamma.$
\end{proof}

\begin{proof}[Proof of Theorem~\ref{Thm:Main_Theorem}] The proof follows from Lemma~\ref{lem:unitofinite}, Lemma~\ref{lem:BinaryApprox} and Proposition \ref{prop:Key_Proposition2}. 
\end{proof} 

\begin{remark}[A comment about the proof of Corollary~\ref{Cor:Bar}]
Corollary~\ref{Cor:Bar} follows simply because the weak topology on $\Mcal_n(\mathbb{C})$ is the same as the operator norm topology. However, a direct proof following our technique in this case would also be instructive. For brevity, we present a brief sketch. Using Lemma~\ref{lem:ApproximationBySimpleUnimod} and~\ref{lem:BinaryApprox}, our task boils down to approximating to unimodular functions taking only two values--the identity matrix, $I$ and some unitary matrix $T$. We know that a unitary matrix can be written as $U^*DU$ where $U$ is unitary and $D$ is a diagonal matrix with entries $\lambda_j\in \mathbb{T}$. Apply the Douglas-Rudin theorem \ref{thm:DR} to obtain inner functions $\varphi_i, \psi_i$ such that the quotient approximates the unimodular functions taking values $1$ and $\lambda_i$. Let $D_{\varphi}, D_{\psi}$ be diagonal matrices whose diagonal entries are $\varphi_i, \psi_i$, respectively. Define the matrix-valued inner functions $$\Phi = U^*D_{\varphi}U\quad \text{ and } \quad \Psi = U^*D_{\psi}U,$$ then $\Phi\Psi^{-1}$ approximates our unimodular function.  
 \end{remark}

\section{Proof of Theorem~\ref{Thm:MainTheorem_Ball2}}
\label{sec:Thm2}
Just like Proposition~\ref{prop:Key_Proposition2}, the proof of Theorem~\ref{Thm:MainTheorem_Ball2} also follows from a similar proposition that we state and prove below. The only difference here is that the approximation now holds only a.e. This restriction comes from the fact that for $d\geq 2$, scalar unimodular functions can only be approximated by the quotients of inner functions in a.e. sense~\cite{RudinBall}. The idea of the proof is essentially the same as that of Proposition~\ref{prop:Key_Proposition2}.

\begin{proposition}
Let $E\subseteq \mathbb{S}_d$ be a measurable subset. Let $f:\mathbb{S}_d\to \Bcal(\Hcal)$ be a unimodular function such that
\begin{equation*}
		f=
		\begin{cases}
			I_{\Hcal} & \text{on }\   E; \\
			T & \text{ on }\   E^c
		\end{cases}
	\end{equation*}
 for some $T\in \Ucal(\Hcal)$. Then, there exists a sequence of two-sided inner functions $\Phi_k, \Psi_k$ such that 
\[\|f(z)-\Phi_k(z)\Psi_k(z)^{-1}\|_{op}\to 0\quad \text{for a.e. } z\in \mathbb{S}_d\;.\]
\end{proposition}
\begin{proof}
The proof closely follows the proof of Proposition~\ref{prop:Key_Proposition2}, we only sketch the proof below highlighting the significant departures.

Consider a functions $f:\mathbb{S}_{d}\to \Bcal(L^2(\Sigma, \mu))$ such that $f\vert_{E}=I$ and $f\vert_{E^c}=M_{\eta}$ where $$\eta = \sum_{i=1}^{m}\alpha_i\chi_{\Sigma_i}.$$ Define 
\begin{equation*}
		\eta_i(z)=
		\begin{cases}
			1 & \text{if }\  z\in E; \\
			\alpha_i & \text{ if }\  z\in E^c
		\end{cases}
	\end{equation*}

 Invoke Theorem $5.5$ of \cite{RudinBall} to get two sequences of inner functions $\varphi_{k, i}, \psi_{k, i}$ such that $$\theta_{k, i}(z)= \varphi_{k, i}(z)\psi_{k, i}^{-1}(z)\to \eta_i(z)$$ as $k\to \infty$ for a.e. $z\in \mathbb{S}_d$. Define $$\Phi_k(z) \coloneqq M_{\varphi_k(z)}\quad \text{ and}  \quad \Psi_k(z) \coloneqq M_{\psi_k(z)},$$ where $$\varphi_k(z)= \sum_{i=1}^{m}\varphi_{k, i}(z)\chi_{\Sigma_i}\quad \text{ and } \quad \psi_k(z)=\sum_{i=1}^{m}\psi_{k, i}(z)\chi_{\Sigma_i}.$$ Then, $\Phi_k$ and $\Psi_k$ are two-sided inner functions. 
Now, we observe that 

$$\|f(z)-\Phi_k(z)\Psi_k(z)^{-1}\|_{op}\leq \max_{i=1}^{m}|\theta_{k, i}(z)-\eta_i(z)|\to 0$$ 

as $k\to \infty$ 
for almost everywhere $z\in \mathbb{S}_d$.
\end{proof}

\section{Discussion}\label{sec:conclusion}

It is reasonable to inquire whether there is any approximation result without the imposition of any unimodular condition, specifically for arbitrary bounded measurable functions. The following proposition addresses this question, providing approximation results applicable to any such function. However, this comes at the cost of one of the involved factors. Unlike Theorem~\ref{Thm:Main_Theorem}, where the factors consist of single two-sided inner functions, one of the current factors now comprises linear combinations of two-sided inner functions.
\begin{proposition}
Let $f$ be a function in $L^\infty(\mathbb{S}_d, \Bcal(\mathcal H))$. Then the following hold: 

\begin{enumerate}
    \item if $d=1$, then $f$ can be approximated uniformly, in weak topology, by the functions of the form $\psi{\varphi}^*$, where $\psi$ is a finite linear combination of two-sided inner functions and $\varphi$ is two-sided inner;
    \item if $d>1$, then the approximation above holds in a.e. (almost everywhere) sense.
\end{enumerate}
 \end{proposition}
\begin{proof}
 The proof is a consequence of our main Theorems. Indeed, let $\mathcal{Q}$ be the set of all functions of the form $\psi\varphi^*$, where $\psi$ is a finite linear combination of two-sided inner functions and $\varphi$ is a two-sided inner function. First, assume that $d=1$. Let $\chi_{E}$ be the characteristic function of a measurable set $E \subset \mathbb{T}$. And, let $U$ be a unitary operator. Note that $(2\chi_E - 1)U$ is  unimodular and hence is in $\overline{\mathcal{Q}}$ by Theorem~\ref{Thm:Main_Theorem}. Since $\overline{\mathcal{Q}}$ is a linear space, it follows that $\chi_EU$ is in $\overline{\mathcal{Q}}$. Since the weak-closure of $\Ucal(\Hcal)$ in $\Bcal(\Hcal)$ is the set $$\{T\in \Bcal(\Hcal): \norm{op}{T}\leq 1\}$$ and $\overline{\mathcal{Q}}$ is a linear space, we conclude that $\chi_{E}T\in \overline{\mathcal{Q}}$ for every $T\in \Bcal(\Hcal)$. It now follows from a standard approximation argument that \[L^\infty(\mathbb{T}, \Bcal(\mathcal H))\subset \overline{\mathcal Q}\;.\] 
 On the other hand, it is obvious that $\overline{\mathcal Q}\subset L^\infty(\mathbb{T}, \Bcal(\mathcal H))$. Thus, we conclude that
 $$\overline{\mathcal{Q}} = L^\infty(\mathbb{T}, \Bcal(\mathcal H)).$$
 
 This completes the proof of part $(1)$. The proof of part $(2)$ follows along the same line where instead of Theorem \ref{Thm:Main_Theorem}, we apply Theorem \ref{Thm:MainTheorem_Ball2}. We leave the details to the interested readers.
\end{proof}

Suppose we have a scalar-valued function $f \in L^\infty$, represented as $f = \overline{h}g$, where $h$ and $g$ are non-zero functions in $H^\infty(\mathbb{D})$. It is straightforward to observe that $f$ is log-integrable, meaning it is non-zero almost everywhere, and $\log |f|$ is integrable ($\log |f| \in L^1$). Douglas and Rudin \cite{RudinDouglas} asked if the converse holds. Initially, they conjectured that it does not. However, Bourgain~\cite{Bo} demonstrated that the log-integrability condition is indeed sufficient. A decade ago, this result was extended to matrix-valued functions by Barclay~\cite{BS}. Both proofs, whether for scalar-valued or matrix-valued functions, crucially relied on the Douglas-Rudin approximation results.

It is worth noting that this factorization result is not completely settled in the operator-valued setting. Recent investigations have explored various factorization results concerning operator-valued inner functions~\cites{curto2022operator, Curto2}, yet a comprehensive understanding remains elusive. We believe that our generalization of the Douglas-Rudin approximation theorem holds promise in resolving Douglas-Rudin type factorization results for the operator-valued setting.


We end with a natural remark about a possible extension of our results. It is natural to wonder if the Douglas-Rudin theorem holds for (operator-valued) unimodular functions defined on $\T^{d}$. Clearly, our proof technique would yield the result, if the Douglas-Rudin theorem for scalar-valued unimodular functions were available for $\T^{d}$. However, to the best of our knowledge, the result available in this direction is only available for the continuous unimodular functions~\cite{Mc}. It would be interesting to explore the possibility of extending the result in~\cite{Mc} to measurable unimodular function--which would then yield the result in operator-valued case as well following our technique.

\bibliographystyle{alpha} 
\bibliography{references}

\end{document}